\documentclass{amsart}
\usepackage{amsmath,amsfonts,amssymb}
\usepackage{amsthm}
\usepackage{array}
\usepackage{graphicx,color}
\usepackage{graphics}
\usepackage{mathrsfs}
\usepackage{hyperref} 
\theoremstyle{plain}
\newtheorem{theorem}{Theorem}
\newtheorem{prop}{Proposition}
\newtheorem{lemma}{Lemma}

\newtheorem{cor}{Corollary}
\newtheorem{thmx}{Theorem}

\theoremstyle{definition}

\newtheorem{remark}{Remark}

\newtheorem{expl}{Example} 
\newcommand{\R}{\mathbb{R}}
\renewcommand{\Re}{\operatorname{Re}}

\newcommand{\NM}{\mathbb{N}}
\newcommand{\CM}{\mathbb{C}}

\newcommand{\Z}{\mathbb{Z}}
\newcommand{\bo}{ {\rm O} }

\newcommand{\dint}{\ds\int}

\newcommand{\ds}{\displaystyle}
\newcommand{\dsum}{\ds\sum}

\newcommand{\eqskip}{ \vspace*{2mm}\\ }
\newcommand{\fr}[2]{\frac{\ds #1}{\ds #2}}

\newcommand{\ttt}{\mathscr{T}}
\newcommand{\DD}{\mathscr{D}}

\def\noi{\noindent}
\def\resp{\emph{resp.}}
\renewcommand{\L}{L}
\renewcommand{\dot}[1]{#1'}
\renewcommand{\ddot}[1]{#1''}
\renewcommand{\d}{\mathrm{d}}
\newcommand{\W}{W}
\def\iy{\infty}
\renewcommand\d{\mathrm{d}}
\newcommand{\veps}{\varepsilon}

\newcommand{\frp}[2]{\frac{\partial #1}{\partial #2}}
\newcommand{\frpp}[2]{{\partial #1}/{\partial #2}}
\renewcommand\bar{\overline}

\newcommand{\re}{{\rm Re}}
\renewcommand\tilde{\widetilde}
\newcommand\ch{\mathrm{ch}}
\newcommand\sh{\mathrm{sh}}
\newcommand\shc{\mathrm{shc}}

\newcommand{\Tr}{{\rm Tr}}

\newcommand{\wt}[1]{\widetilde{#1}}
\newcommand{\aalpha}{\alpha}
\newcommand{\balpha}{q}
\newcommand{\bbeta}{r}

        \definecolor{purple}{rgb}{0.4,0.2,1}

\title[Maximal determinants of Schr\"{o}dinger operators]%
{Maximal determinants of Schr\"{o}dinger operators on bounded intervals}
\author{Clara L.~Aldana, Jean-Baptiste Caillau and Pedro Freitas}
\address{Universidad del Norte, V\'\i a Puerto Colombia, Barranquilla, Colombia}
\email{claldana@uninorte.edu.co}
\address{Universit\'e C\^ote d'Azur, CNRS, Inria, LJAD, France}
\email{jean-baptiste.caillau@univ-cotedazur.fr}
\address{Departamento de Matem\'{a}tica, Instituto Superior T\'{e}cnico, Universidade de Lisboa, Av. Rovisco Pais, 1049-001 Lisboa, Portugal
\& Grupo de F\'{\i}sica Matem\'{a}tica, Faculdade de Ci\^{e}ncias, Universidade de Lisboa, Campo Grande, Edif\'{\i}cio C6,
1749-016 Lisboa, Portugal }
\email{psfreitas@fc.ul.pt}
\date{\today}
\thanks{Part of this work was done during a sabbatical leave of the second and third
authors at the Laboratoire Jacques-Louis Lions, Sorbonne Universit\'{e} \& CNRS, whose
hospitality is gratefully acknowledged. The first author was supported by ANR grant
ACG: ANR-10-BLAN 0105 and by the Fonds National de la Recherche, Luxembourg 7926179.
The second and third authors acknowledge support from FCT, grant no.~PTDC/MAT-CAL/4334/2014.}

\begin{document}

\begin{abstract} We consider the problem of finding extremal potentials for the functional determinant
of a one-dimensional Schr\"{o}dinger operator defined on a bounded interval with Dirichlet boundary conditions
under an $L^\balpha$-norm res\-triction ($\balpha\geq 1$). This is done by first extending the definition of the functional
determinant to the case of $L^\balpha$ potentials and showing the resulting problem to be equivalent to a problem
in optimal control, which we believe to be of independent interest. We prove existence, uniqueness and describe some
basic properties of solutions to this problem for all $\balpha\geq 1$, providing a complete characterization of extremal potentials in
the case where $\balpha$ is one (a pulse) and two (Weierstrass's $\wp$ function).
\end{abstract}

\maketitle


\section*{Introduction}
\noi An important quantity arising in connection with self-adjoint elliptic operators is the functional (or spectral)
determinant. This has been applied in a variety of settings in mathematics and in physics, and is based on the regularisation of the
spectral zeta function associated to an operator $\ttt$ with discrete spectrum. This zeta function is defined by
\begin{equation}\label{zetaT}
 \zeta_{\ttt}(s) = \dsum_{n=1}^{\infty} \lambda_{n}^{-s},
\end{equation}
where the numbers $\lambda_{n},  \, n=1,2,\ldots$ denote the eigenvalues of $\ttt$ and,
for simplicity, and without loss of generality from the perspective of this work as we will see below, we shall assume that these
eigenvalues are all positive and with finite multiplicities.
Under these conditions, and for many operators such as the Laplace or Schr\"{o}dinger operators, the above series will be convergent
on a right half-plane, and may typically be extended meromorphically to the whole of $\CM$. Furthermore, zero
is not a singularity and since, formally,
\[
 \zeta_{\ttt}'(0) = -\dsum_{n=1}^{\infty} \log(\lambda_{n}),
\]
the regularised functional determinant is then defined by
\begin{equation}\label{detdef}
 \det\ttt = e^{-\zeta_{\ttt}'(0)},
\end{equation}
where $\zeta_{\ttt}'(0)$ should now be understood as referring to the meromorphic extension mentioned above. This
quantity appears in the mathematics and physics literature in connection to path integrals, going back at least to the early 1960's. Examples
of calculations of determinants for operators with a potential in one dimension may be found in~\cite{BFK,geya,Levit-Smilansky} and, more recently,
for the harmonic oscillator in arbitrary dimension~\cite{frei}. Some of the regularising techniques for zeta functions which are needed in order
to define the above determinant were studied in~\cite{mipl}, while the actual definition~\eqref{detdef} was given in~\cite{rasi}. 
Within such a context, it is then natural to study extremal properties of these global spectral objects and this
question has indeed been addressed by several authors, mostly when the underlying setting is of a geometric nature~\cite{aursal,ops,aar-r}.

In this paper, we shall consider the problem of optimizing the functional determinant for a Schr\"{o}dinger operator defined on a bounded interval
together with Dirichlet boundary conditions. More precisely, let $\ttt_V$ be the operator associated with the eigenvalue problem defined by
\begin{equation}
 \label{schrodinger}
 \left\{
 \begin{array}{l}
 -\phi''+V\phi=\lambda \phi\eqskip
 \phi(0)=\phi(1)=0,
 \end{array}
 \right.
\end{equation}
where $V$ is a potential in $L^\balpha[0,1]$ $(\balpha\geq 1)$. For a given $\balpha$ and a positive constant $A$, we are interested
in the problem of optimizing the determinant given by
\begin{equation}\label{prob}
 \det\ttt_V \to \max,\quad \|V\|_\balpha \leq A 
\end{equation}
where $\| \cdot\|_\balpha$ denotes the norm on $L^\balpha[0,1]$.
For smooth bounded potentials the determinant of such operators is known in closed form and actually requires no computation of the
eigenvalues themselves. In the physics literature such a formula is sometimes referred to as the Gelfand-Yaglom formula and
a derivation may be found in~\cite{Levit-Smilansky}, for instance---see also~\cite{BFK}. More precisely, for the operator $\ttt_V$ defined
by~\eqref{schrodinger} we have $\det \ttt_V = 2y(1)$, where $y$ is the solution of the initial value problem
\begin{equation}
 \label{controlproblem}
 \left\{
 \begin{array}{l}
 -y''+Vy=0\eqskip
 y(0)=0, \;\; y'(0)=1.
 \end{array}
 \right.
\end{equation}

\noi We shall show that this expression for the determinant still holds for $L^\balpha$ potentials and
study the problem defined by~\eqref{prob}. We then prove that~\eqref{prob} is well-posed and has a unique solution for all
$\balpha\geq 1$ and positive $A$.
In our first main result we consider the $\L^1$ case where the solution is given by
a piecewise constant function.

\begin{thmx}[Maximal $L^1$ potential] \label{thm1}
Let $\balpha=1$. Then for any positive number $A$ the unique solution to problem~\eqref{prob}
is the symmetric potential given by
\[
V_{A}(x) = \fr{A}{\ell(A)} \chi_{\ell(A)},
\]
where $\chi_{\ell(A)}$ denotes the characteristic function of the interval of length
\[
 \ell(A) = \fr{A}{(1+\sqrt{1+A})^2}
\]
centred at $1/2$. The associated maximum value of the determinant is
\[
\max_{\| V\|_{1}=A } \det\ttt_V = \frac{4}{1+\sqrt{1+A}}\exp\left( \frac{A}{1+\sqrt{1+A}} \right).
\]
\end{thmx}

\noi In the case of general $\balpha$ we are able to provide a similar result concerning existence and uniqueness, but the corresponding
extremal potential is now given as the solution of a second order (nonlinear) ordinary differential equation.

\begin{thmx}[Maximal $L^\balpha$ potential, $\balpha>1$]  \label{thm2}
For any $\balpha>1$ and any positive number $A$, there exists a unique solution to problem~\eqref{prob}.
This maximal potential is given by
\[
V_{A} = \frac{\balpha}{4\balpha-2}\sqrt[\balpha-1]{\Psi},
\]
where $\Psi$ is the solution to
\[ \Psi''-|\Psi|^\alpha+2H = 0,\quad \alpha:=\balpha/(\balpha-1), \]
\[ \Psi(0)=0,\quad \Psi'(0)=H-c(A,\balpha). \]
Here $c(A,\balpha):=(1/2)(A(4\balpha-2)/\balpha)^\balpha$, and $H$ is a
(uniquely defined) constant satisfying $H>c(A,\balpha)$.
The function $\Psi$ is non-negative on $[0,1]$, and
the maximal potential is symmetric with respect to $t=1/2$, smooth on $(0,1)$,
strictly increasing on $[0,1/2]$,
with zero derivatives at $t=0$ and $t=1$ if $1<\balpha<2$, positive derivative at $t=0$
(\resp\ negative derivative at $t=1$) if $\balpha=2$, and vertical tangents at both
endpoints if $q>2$.
\end{thmx}

\noi The properties given in the above theorem provide a precise qualitative description of the evolution of maximal potentials as
$q$ increases from $1$ to $+\infty$. Starting from a rectangular pulse ($q=1$), solutions become regular for $q$ on $(1,2)$, having
zero derivatives at the endpoints. In the special case of $L^{2}$, the maximising
potential can be written in terms of the Weierstrass elliptic function
and has finite nonzero derivatives at the endpoints. This marks the transition to potentials with singular derivatives at the boundary
for $\balpha$ larger than two, converging towards an optimal constant potential in the limiting $L^{\infty}$ case.

\begin{thmx}[Maximal $L^2$ potential]  \label{thm3}
Let $\balpha=2$. Then for any positive number $A$ the unique solution to problem~(\ref{prob})
is given by
\[ V_A(t) = \frac{1}{3}\wp\left(\frac{2t-1}{2\sqrt{6}}+\omega'\right),\quad t \in [0,1], \]
where $\wp$ is the Weierstrass elliptic function associated to invariants
\[ g_2 = 24 H,\quad g_3 = -6(H-9A^2/2)^2, \]
and where $\omega'$ is the corresponding
imaginary half-period of the rectangular lattice of
periods. The corresponding (unique) value of $H$ such that
\[ \wp\left(\frac{1}{2\sqrt{6}}+\omega'\right) = 0 \]
is in $(9A^2/2,h^*(A))$, where $h^*(A)$ is the unique root of the polynomial $128 H^3-9(H-9A^2/2)^4$
in $(9A^2/2,\iy)$.
\end{thmx}

The paper is structured as follows. In the next section we show that the functional determinant of Schr\"{o}dinger
operators with Dirichlet boundary conditions on bounded intervals and with potentials in $L^\balpha$ is well defined and we extend the formula
from~\cite{Levit-Smilansky} to this general case. The main properties of the determinant, namely boundedness and monotonicity over
$L^\balpha$, are studied in Section~\ref{s2}. Having established these, we then consider the optimal control problem~\eqref{prob} of
maximising $y(1)$ in~\eqref{controlproblem} in Sections~\ref{s3} and~\ref{s4}, where the proofs of our main results Theorems~\ref{thm1}
and~\ref{thm2}-\ref{thm3} are given, respectively. 


\section{The determinant of one-dimensional Schr\"{o}dinger operators\label{s1}}

\noi We consider the eigenvalue problem defined by~\eqref{schrodinger} associated to
the operator $\ttt_V$ on the interval $[0,1]$ with potential $V\in L^\balpha[0,1]$ for $\balpha\geq 1$. Although no further
restrictions need to be imposed on $V$ at this point, for our purposes it will be sufficient to consider $V$ to be non-negative,
as we will show in Proposition~\ref{positivitythm} below. This simplifies slightly the definition of the associated zeta function given by~\eqref{zetaT}
and thus also that of the determinant. Hence, in the rest of this section we assume that $V$ is a non-negative potential.

For smooth potentials, and as was already mentioned in the Introduction, it is known that the regularised functional determinant of $\ttt_V$ is
well defined~\cite{BFK,Levit-Smilansky}, and this has been extended to potentials with specific singularities~\cite{Lesch,Lesch-Tolk}.
We shall now show that this is also the case for general potentials in $L^\balpha[0,1]$, for $\balpha\geq 1$. We first show that the
zeta function associated with the operator $\ttt_V$ as defined by~\eqref{zetaT} is analytic at the origin and has as its only singularity
in the half-plane $\re(s)>-1/2$ a simple pole at $1/2$. This is done adapting one of the approaches originally used by Riemann (see~\cite{riem}
and also~\cite[p.~21ff]{titc}). We then show that the method from~\cite{Levit-Smilansky} can be used to prove that the determinant is still given by $2y(1)$,
where $y$ is the solution of the initial value problem~\eqref{controlproblem}.

In order to show these properties of the determinant, it is useful to consider the heat trace associated with $\ttt_V$, defined by
\[
\Tr(e^{-t \ttt_V}) = \sum_{n=1}^{\infty} e^{-t \lambda_n}.
\]
Our first step is to show that the behaviour of the heat trace as $t$ approaches zero for any non-negative potential in $L^{\balpha}(0,1)$ is
the same as in the case of smooth potentials.
\begin{prop}\label{heattrace}
Let $\ttt_V$ be the Schr\"odinger operator defined by problem
(\ref{schrodinger}) with $V\in L^{\balpha}[0,1]$, then 
$$\Tr(e^{-t \ttt_V}) = \frac{1}{2\sqrt{\pi t}} - \frac{1}{2} + \bo(\sqrt{t}), \quad \textrm{ as } \quad t\to 0.$$
\label{detSchL1}
\end{prop}
\begin{proof}
For potentials which are the derivative of a function of bounded variation it was proved in Section $3$ in~\cite{SavcShka99}
(see also~\cite[Theorem 1]{Savchuk}) that the eigenvalues behave asymptotically as
\begin{equation}
\lambda_n = n^2\pi^2 + \bo(1), \quad n=1,2,\dots \label{eq:evas}
\end{equation}
when $n$ goes to $\infty$. For a potential in $L^{\balpha}$ we may thus assume the above asymptotics which imply
the existence of a positive constant $c$ such that
\[
 \pi^2 n^2 - c \leq \lambda_n \leq \pi^2 n^2 + c
\]
uniformly in $n$.
For the zero potential the spectrum is given by $\pi^2 n^2$ and the heat trace associated with it becomes the Jacobi theta function defined by
\[
 \psi(t) = \sum_{n=1}^{\infty} e^{-n^2 \pi^2 t}.
\]
We are interested in the behaviour of the heat trace for potentials $V$ for small positive $t$, and we will determine this
behaviour by comparing it with that of $\psi$. For simplicity, in what follows we write
\[
 \varphi(t) = \Tr(e^{-t \ttt_V}).
\]
We then have
\[
\left(e^{-c t}-1\right) \psi(t) \leq \varphi(t) - \psi(t) \leq \left(e^{c t}-1\right) \psi(t),
\]
and, since $\psi$ satisfies the functional equation~\cite[p. 22]{titc}
\[
\psi(t)=\frac{1}{2\sqrt{\pi t}} - \frac{1}{2} + \frac{1}{\sqrt{\pi t}} \ \psi\left(\frac{1}{\pi^2 t}\right),
\]
it follows that
\begin{multline*}
\frac{1}{\sqrt{\pi t}} \ \psi\left(\frac{1}{\pi^2 t}\right) + (e^{-c t}-1)\psi(t) \leq \varphi(t) - \frac{1}{2\sqrt{\pi t}} + \frac{1}{2}
\leq \frac{1}{\sqrt{\pi t}} \ \psi\left(\frac{1}{\pi^2 t}\right) + (e^{c t}-1) \psi(t).\\
\end{multline*}
Since $\fr{1}{\sqrt{\pi t}} \ \psi\left(\fr{1}{\pi^2 t}\right) = \bo(e^{-C/t})$ for some $C>0$ and 
\[
 (e^{c t}-1) \psi(t) = \bo(\sqrt{t})
\]
as $t\to 0$, it follows that
\[
 \varphi(t) =  \frac{1}{2\sqrt{\pi t}} - \frac{1}{2} + \bo(\sqrt{t}) \mbox{ as } t\to 0.
\]\end{proof}
\begin{remark}
 Note that although the heat trace for the zero potential $\psi$ satisfies
\[
 \psi(t)=\frac{1}{2\sqrt{\pi t}} - \frac{1}{2} + \bo(t^{\alpha}) \mbox{ as } t\to 0^{+}
\]
for any positive real number $\alpha$, this will not be the case for general potentials, where we can only ensure that the next term in the expansion
will be of order $\sqrt{t}$.
\end{remark}

\noi We may now consider the extension of $\zeta_{\ttt_V}$ to a right half-plane containing the origin.
\begin{prop}\label{meromorphicext}
 The spectral zeta function associated with the operator $\ttt_V =-\Delta + V$ with Dirichlet boun\-dary conditions and potential $V\in L^{\balpha}[0,1]$, $\balpha\geq 1$,  
 defined by~\eqref{zetaT} may be extended to the half-plane $\re(s)>-1/2$ as a meromorphic function with a simple pole at $s=1/2$, whose residue is
 given by $1/(2\pi)$. 
\end{prop}
\begin{proof}
 We start from
\[
 \dint_{0}^{+\infty} t^{s-1} e^{-\lambda_{n}t} {\rm d}t = \Gamma(s) \lambda_{n}^{-s}
\]
which is valid for $\re(s)>0.$ Summing both sides in $n$ from one to infinity yields
\[
 \zeta_{\ttt_V}(s) = \fr{1}{\Gamma(s)} \dint_{0}^{+\infty} t^{s-1} \varphi(t) {\rm d}t,
\]
where $\varphi$ denotes the heat trace as above and the exchange between the sum and the integral is valid for $\re(s)>1/2$.
By Proposition~\ref{heattrace} we may write
\[
 \varphi(t) = \frac{1}{2\sqrt{\pi t}} - \frac{1}{2} + f(t)
\]
where $f(t) = \bo(\sqrt{t})$ as $t$ approaches zero. Since, in addition, $\varphi(t) = \bo(e^{-ct})$ as $t$ approaches infinity, for some $c>0$, we have
\[
\begin{array}{lll}
 \zeta_{\ttt_V}(s) & = & \fr{1}{\Gamma(s)}\left[ \dint_{0}^{1} \left( \frac{1}{2\sqrt{\pi t}} - \frac{1}{2} + f(t) \right) t^{s-1}  {\rm d}t +
 \dint_{1}^{+\infty} \varphi(t) t^{s-1} {\rm d}t\right]\eqskip
 & = & \fr{1}{\Gamma(s)}\left( \fr{\pi^{-1/2}}{2s-1} -\fr{1}{2s} \right) + F(s)
\end{array}
\]
which is valid for $\re(s)>1$ and where $F$ is an analytic function in the half-plane $\re(s)>-1/2$. Due to the simple zero of $1/\Gamma(s)$
at zero we see that the expression in the right-hand side is well defined and meromorphic in the half-plane $\re(s)>-1/2$, except for the
simple pole at $s=1/2$, showing that we may extend $\zeta_{\ttt_V}(s)$ to this half-plane. The value of the residue is obtained by a
standard computation.
\end{proof}
\begin{remark}
 It is clear from the proof that the behaviour on $\re(s)\leq 1/2$ will depend on the potential $V$. This may be seen from the simple example of a constant
 potential $V(x) \equiv a$, for which the heat trace now satisfies
 \[
  \varphi(t) = e^{-a t}\psi(t).
 \]
From this we see that when $a$ is nonzero $\varphi(t)$ has an expansion for small $t$ with terms of the form $t^{-(n-1)/2}$, for all non-negative
integers $n$. These terms produce simple poles at $s=-(2n+1)/2$, $n\geq 0$ an integer, with residues depending on $a$. When $a$ vanishes
we recover $\zeta_{\ttt_V}(s) = \zeta(2s)$ and there are no poles other than the simple pole at $s=1/2$.
\end{remark}

We are now ready to extend the result in~\cite{Levit-Smilansky} to the case of
$L^{\balpha}$ potentials with $\balpha$ greater than or equal to one.
\begin{theorem} \label{thm0}
The determinant of the operator $\ttt_V =-\Delta + V$ with Dirichlet boun\-dary
conditions and potential $V\in L^\balpha[0,1]$, $\balpha \in[1,+\infty]$,
is given by
\begin{equation}\label{eq:dety1}
 \det\ttt_V = 2y(1),
\end{equation}
where $y$ is the solution of the initial value problem~\eqref{controlproblem}.
\end{theorem}
\begin{proof}
We shall follow along the lines of the proof in~\cite[Theorem 1]{Levit-Smilansky} for smooth potentials,
which consists in building a one-parameter family of potentials, $\aalpha V$, connecting the zero potential,
for which the expression for the determinant may be computed explicitly, with the potential $V$, and comparing
the way these two quantities change.
More precisely, the main steps in this approach are as follows. For $\aalpha\in [0,1]$, we define the family of
operators $\ttt_{\aalpha}$ in $L^{2}[0,1]$ by $\ttt_{\aalpha} u(\aalpha, t) :=  -u''(\aalpha,t) + \aalpha V(t) u(\aalpha,t)$
and consider the eigenvalue problem
\begin{equation}\label{eigprob}
  \left\{
 \begin{array}{l}
\ttt_{\aalpha} u(\aalpha, t)=\lambda(\aalpha) u(\aalpha,t)\eqskip
 u(\aalpha,0)=0, \;\; u(\aalpha,1)=0
 \end{array}
 \right.
\end{equation}
with solutions $\lambda_k(\aalpha)$, $u_k(\aalpha)$.
We have that $\{\ttt_{\aalpha}\}_{\aalpha\in [0,1]}$ is an analytic
family in $\aalpha$ of type $B$ in the sense of Kato. This follows from~\cite[Example
VII.4.24]{Kato}, which covers a more general case. Then, by Remark VII.4.22 and Theorem
VII.3.9 in~\cite{Kato} we obtain that the eigenvalues $\lambda_k(\aalpha)$ and its associated (suitably normalized) eigenfunction $u_k(\aalpha)$, for any
$k\geq 1$, are analytic functions of $\aalpha$. 
The corresponding $\zeta$-function is given by
\begin{equation}\label{zetaalpha}
 \zeta_{\ttt_{\aalpha}}(s) = \dsum_{j=1}^{\infty} \lambda_{j}^{-s}(\aalpha).
\end{equation}
Although this series is only defined for $\Re(s)>1/2$, we know from Proposition-\ref{meromorphicext} that the spectral zeta function defined by it
can be extended uniquely to a meromorphic function on the half-plane $\re(s)>-1/2$ which is analytic at zero.
We shall also define $y(\aalpha,t)$ to be the family of solutions of the initial value problem
\begin{equation*}
 \left\{
 \begin{array}{l}
 \ttt_{\aalpha} y(\aalpha,t)=0\eqskip
 y(\aalpha,0)=0, \;\; y'(\aalpha, 0)=1.
 \end{array}
 \right.
\end{equation*}
By Proposition~\ref{prop10} in the next section the quantity $2y(\aalpha,1)$ is well defined for $V \in L^{\balpha}[0,1]$.
The idea of the proof is to show that  
\begin{equation}\label{idlog}
 \frac{\d}{\d\aalpha} \log\left( \det\ttt_{\aalpha} \right) =
 \frac{\d}{\d\aalpha} \log\left( y\left(\aalpha,1\right) \right)
\end{equation}
for $\aalpha\in [0,1]$. Since at $\aalpha=0$, $\det\ttt_0=2 y(0,1)$, it follows that equality of the two functions will still hold
for $\aalpha$ equal to one. The connection between the two derivatives is made through the Green's function of the operator. 
We shall first deal with the left-hand side of identity~\eqref{idlog}, for which we need to differentiate the series defining the spectral
zeta function with respect to both $\aalpha$ and $s$, and then take $s=0$. We begin by differentiating the series in equation~\eqref{zetaalpha}
term by term with respect to $\aalpha$ to obtain
\begin{equation}\label{dzetadalpha}
 \frac{\ds \partial}{\ds \partial \aalpha}\zeta_{\ttt_{\aalpha}}(s) = -s\dsum_{j=1}^{\infty} \frac{\ds \lambda_{j}'(\aalpha)}{\ds \lambda_{j}^{s+1}(\aalpha)},
\end{equation}
where the expression for the derivative of $\lambda_{j}$ with respect to $\aalpha$ is given by
$$ \lambda_j'(\aalpha) = \left(\int_0^1 V(x) u_{j}^2(\aalpha,x) \,\d x\right)
\left(\int_0^1 u_{j}^2(\aalpha,x) \,\d x\right) ^{-1}. $$
For potentials in $L^{1}[0,1]$ we have, as we saw above, that the corresponding eigenvalues of problem~\eqref{eigprob}
satisfy the asymptotics given by~\eqref{eq:evas},
while the corresponding eigenfunctions satisfy
\[
 u_{j}(x) = \sin(j\pi x) + r_{j}(x),
\]
where $r_{j}(x) =  {\rm O}(1/j)$ uniformly in $x$. We thus have
\[
\begin{array}{lll}
 \dint_{0}^{1} u_{j}^{2}(x) \,\d x & = & \dint_{0}^{1} \sin^{2}(j \pi x) \,\d x + 2
 \dint_{0}^{1} r_{j}(x)\sin(j \pi x) \,\d x
 +  \dint_{0}^{1} r_{j}^{2}(x) \,\d x\eqskip
 & = & \frac{\ds 1}{\ds 2} + 2\dint_{0}^{1} r_{j}(x)\sin(j \pi x) \,\d x +
 \dint_{0}^{1} r_{j}^2(x) \,\d x\eqskip
 & = & \frac{\ds 1}{\ds 2} + {\rm O}(1/j).
 \end{array}
\]
In a similar way, the numerator in the expression for $\lambda_{j}'(\aalpha)$ satisfies
\[
\begin{array}{lll}
 \dint_{0}^{1} V(x) u_{j}^{2}(x) \,\d x & = & \dint_{0}^{1} V(x)\sin^{2}(j \pi x) \,\d
 x + 2  \dint_{0}^{1} V(x) r_{j}(x)\sin(j \pi x) \,\d x\eqskip
 & & \hspace*{1cm} +  \dint_{0}^{1}  V(x)r_{j}^{2}(x) \,\d x\eqskip
 & = & \frac{\ds 1}{\ds 2}\dint_{0}^{1} V(x)\left[1-\cos(2j\pi x)\right] \,\d x + {\rm O}(1/j)\eqskip
 & = & \frac{\ds 1}{\ds 2}\dint_{0}^{1} V(x) \,\d x - \frac{\ds 1}{\ds 2}\dint_{0}^{1}
 V(x)\cos(2j\pi x) \,\d x + {\rm O}(1/j)\eqskip
 & = & \frac{\ds 1}{\ds 2}\dint_{0}^{1} V(x) \,\d x  + {\rm O}(1/j),
\end{array}
 \]
where the last step follows from the Riemann-Lebesgue Lemma. Combining the two asymptotics we thus have

\[
 \lambda_{j}'(\aalpha) = \dint_{0}^{1} V(x) \,\d x + {\rm o}(1)
\]
and so the term in the series~\eqref{dzetadalpha} is of order ${\rm O}(j^{-2s-2})$. This series is thus absolutely convergent (and uniformly convergent in
$\aalpha$) for $\Re(s)>-1/2$. This justifies the differentiation term by term, and also makes it possible to now differentiate it with respect to $s$
to obtain
\[
  \frac{\ds \partial^2}{\ds \partial s\partial\aalpha }\zeta_{\ttt_{\aalpha}}(s) = \dsum_{j=1}^{\infty}\left[
  -1+ s \log\left(\lambda_{j}(\aalpha)\right)\right]\frac{\ds\lambda_{j}'(\aalpha)}{\lambda_{j}^{s+1}(\aalpha)},
\]
which is uniformly convergent for $s$ in a neighbourhood of zero and $\aalpha$ in $[0,1]$. We thus obtain
\[
 \left.\frac{\ds \partial}{\ds \partial \aalpha}\frac{\ds\partial}{\ds \partial s}\zeta_{\ttt_{\aalpha}}(s)\right|_{s=0}=
 -\dsum_{j=1}^{\infty}\frac{\ds\lambda_{j}'(\aalpha)}{\lambda_{j}(\aalpha)}
\]
and
\[
\begin{array}{lll}
 \frac{\ds\d}{\ds\d\aalpha} \log\left[ \det\ttt_{\aalpha}\right] & = &
 \dsum_{j=1}^{\infty}\frac{\ds\lambda_{j}'(\aalpha)}{\lambda_{j}(\aalpha)}\eqskip
 & = & \dsum_{j=1}^{\infty}\lambda_j(\aalpha)^{-1}  \int_0^1
 \frac{u_{j}(\aalpha,x)^2}{\|u_j(\aalpha)\|^2} \ V(t) \,\d t\eqskip
 & = & \dint_0^1 \dsum_{k=1}^{\infty} \lambda_k(\aalpha)^{-1}
 \frac{u_{k}(\aalpha,t)}{\|u_k(\aalpha)\|} \frac{u_{k}(\aalpha,t)}{\|u_k(\aalpha)\|} V(t)
 \,\d t\eqskip
 & = & \dint_0^1 G_{\aalpha}(\lambda=0,t,t) V(t) \,\d t \eqskip
 & = & \Tr(\ttt_{\aalpha}^{-1} V).
\end{array}
\]
Here $G_{\aalpha}(\lambda=0,t,t)$ is the restriction to the diagonal of the Green's function of the boundary value problem~\eqref{eigprob}
at $\lambda=0$. The exchange between the integral and the summation may be justified as above.
We will now consider the right-hand side in~\eqref{idlog}. Here we follow exactly the same computation as in \cite{Levit-Smilansky}.
For that we consider 
$$z(\aalpha,t)= \frac{\d}{\d\aalpha}y(\aalpha,t).$$
Then $z(\aalpha,t)$ is a solution to the initial value problem:
\begin{equation*}
 \left\{
 \begin{array}{l}
 \ttt_{\aalpha} z(\aalpha,t)=-V(t) y(\aalpha,t)\eqskip
 z(\aalpha,0)=0, \;\; (dz/dt) (\aalpha, 0)=0.
 \end{array}
 \right.
\end{equation*}
Using the variation of parameters formula, the solution of this problem is given by
$$z(\aalpha,t)= \frac{1}{W} \left[y(\aalpha,t) \dint_0^t y(\aalpha,r)
\wt{y}(\aalpha,r)V(r) \,\d r 
- \wt{y}(\aalpha,t) \int_0^t y(\aalpha,r) y(\aalpha,r)V(r) \,\d r \right],$$
where the Wronskian $W = y(\aalpha,t) \frac{\d\wt{y}}{\d t}(\aalpha,t)
-\wt{y}(\aalpha,t) \frac{\d y}{\d t}(\aalpha,t)$ is constant, and 
$\wt{y}(\aalpha,t)$ is the solution to the adjoint problem
\begin{equation*}
 \left\{
 \begin{array}{l}
 \ttt_{\aalpha} \wt{y}(\aalpha,t) =0\eqskip
 \wt{y}(\aalpha,1)=0, \;\; (d\wt{y}/dt) (\aalpha, 1)= 1.
 \end{array}
 \right.
\end{equation*}
Therefore we obtain
$$z(\aalpha,1) = y(\aalpha,1) \frac{1}{W} \int_0^1 y(\aalpha,r) \wt{y}(\aalpha,r)V(r)
\,\d r = 
y(\aalpha,1) \int_0^1 G_{\aalpha}(\lambda=0,r,r) V(r) \,\d r,$$
from which the identity~\eqref{idlog} follows. Integrating this with respect to $\aalpha$ yields
$$ \det\ttt_{\aalpha} = c y(\aalpha,1),$$
where $c$ is a constant independent of $\aalpha$. Since $\det\ttt_{0} = 2 y(0,1)$, the result follows.
\end{proof}

We shall finish this section with the example of the pulse potential, of which the optimal potential in the $L^1$ case is
a particular case.
\begin{expl}
Let $S=[x_1,x_2]\subseteq [0,1]$ and $m>0$. A long but straightforward computation shows that the solution of 
$$-y'' + m \chi_S y = 0, \quad \quad y(0)=0, \quad y'(0)=1$$
is given by
$$y(t)= \begin{cases}
t & \ t \in  [0,x_1]\\
a e^{\sqrt{m}t} + b e^{-\sqrt{m}t} & \ t \in [x_1,x_2],\\
c t + d & \ t \in [x_2,1],
\end{cases}
$$ with
\begin{eqnarray*}
a = \frac{1}{2} \left(x_1 + \frac{1}{\sqrt{m}}\right) e^{-\sqrt{m}x_1},  &  &  
b = \frac{1}{2} \left(x_1 - \frac{1}{\sqrt{m}}\right) e^{\sqrt{m}x_1},\\
c = \sqrt{m}  \left(a e^{\sqrt{m}x_2} -b e^{-\sqrt{m}x_2}\right) ,  &  & 
d = a e^{\sqrt{m}x_2} \left(1-\sqrt{m}x_2\right) + b e^{-\sqrt{m}x_2}\left(1+\sqrt{m}x_2\right)
\end{eqnarray*}
\noi Therefore the functional determinant of the operator 
$\ttt = -d^2/dx^2 + m \chi_S$ with Dirichlet boundary conditions is given by
\begin{multline*}
\det(\Delta + m \chi_R) = 2 y(1) = e^{-\sqrt{m}x_1} e^{\sqrt{m}x_2} (x_1 + \frac{1}{\sqrt{m}}) (1 + \sqrt{m} - \sqrt{m}x_2) \\
+ e^{\sqrt{m}x_1} e^{-\sqrt{m}x_2} (x_1 - \frac{1}{\sqrt{m}}) (1 - \sqrt{m} + \sqrt{m}x_2). 
\end{multline*}
\end{expl}

\section{Some properties of the determinant} \label{s2}
\noi Let us denote $\DD$ the operator mapping a potential $V$ in $L^{\balpha}[0,1]$ to 
$\DD_V := y(1)$, where $y$ is the solution of
\[ -y''+Vy=0,\quad y(0)=0, \quad y'(0)=1. \]

\begin{prop} \label{prop10} The operator $\DD$
is well defined, Lipschitz on bounded sets of $L^{\balpha}$ (hence continuous), and
non-negative for non-negative potentials.
\end{prop}

\begin{proof} For $V \in L^{\balpha}[0,1]$,
local existence and uniqueness holds by Caratheodory for $\dot{x}=C(V)x$ with
\begin{equation} \label{eqb1}  
  C(V) := \left[ \begin{array}{cc}
  0 & 1\\
  V & 0 \end{array} \right],\quad x:=(y,\dot{y}),\quad x(0)=(0,1).
\end{equation}
Gronwall's lemma implies that the solution is defined up to $t=1$ and that
\begin{equation} \label{eqb3}
  |x(1)| \leq \exp(1+\|V\|_1).
\end{equation}
Let $V_1$ and $V_2$ in $L^{\balpha}[0,1]$ be both of norm less or equal to $A$.
Denote $x_1$ and $x_2$ the corresponding solutions defined as in (\ref{eqb1}). One has
\[ \dot{x}_1-\dot{x}_2 = (C(V_1)-C(V_2))x_1 + C(V_2)(x_1-x_2), \]
integrating one obtains
\[ |x_1(t)-x_2(t)| \leq \|x_1\|_\iy \|V_1-V_2\|_1
 + \int_0^t |C(V_2(s))|\cdot|x_1(s)-x_2(s)|\,\d s. \]
By the integral version of Gronwall's inequality we obtain
\begin{eqnarray*}
  |x_1(1)-x_2(1)| & \leq & \|x_1\|_\iy \|V_1-V_2\|_1 \exp(1+\|V_2\|_1), \\
  & \leq & e^{2(1+A)} \|V_1-V_2\|_1,
\end{eqnarray*}
implying that $\DD$ is Lipschitz on bounded sets. 
Let finally $V$ be non-negative, and assume by contradiction
that the associated $y$ first vanishes at $\bar{t} > 0$.
As $y(0)=y(\bar{t})=0$ and $\dot{y}(0)>0$,
the function must have a positive maximum at some $\tau \in (0,\bar{t})$.
The function $y$ being
continuously differentiable, $\dot{y}(\tau)=0$. Now,
\[ \dot{y}(\tau) = \dot{y}(0)+\int_0^{\tau}\ddot{y}(t)\,\d t, \]
while $\ddot{y}=Vy \geq 0$ since $V \geq 0$ and $y \geq 0$ on $[0,\tau] \subset
[0,\bar{t}]$. Then $\dot{y}(\tau) \geq 1$, contradicting the definition of
$\tau$.\end{proof}

\noi Being Lipschitz on bounded sets,
the operator $\DD$ sends Cauchy sequences in $L^{\balpha}[0,1]$
to Cauchy sequences in $\R$ (Cauchy-continuity). So its restriction to the dense
subset of smooth functions has a unique continuous extension to the whole space. As
this restriction is equal to the halved
determinant whose definition for smooth potentials is recalled in Section~\ref{s1} for
the operator $-\Delta+V$ with Dirichlet boundary conditions,
$\DD_V$ is indeed the unique continuous extension of
this determinant to $L^\balpha[0,1]$ and, in agreement with Theorem~\ref{thm0},
\[ \DD_V = \frac{1}{2}\det\ttt_V. \]
We begin by proving a uniform upper bound on the maximum value of $y(1)$ for the control problem~\eqref{controlproblem}, and thus for
the determinant of the original Schr\"{o}dinger operator given by~\eqref{schrodinger}.

\begin{prop}\label{thmupperbound}
Assume the potential $V$ is in $L^{\balpha}[0,1]$, $\balpha \in [1,\infty]$.
Then
\[
  \left| \DD_V-1\right|\leq \dsum_{m=1}^{\infty} \frac{\ds \|V\|_{1}^{m}}{\ds (m+1)^{m+1}}.
\]
\end{prop}
\begin{proof}
To prove the proposition it is enough to treat the case $\balpha=1$.
The initial value problem given by equation~\eqref{controlproblem} is equivalent to the integral equation
 \[
  y(t) = t + {\ds\int_{0}^{t}}(t-s)V(s)y(s) \, \mbox{ds}.
 \]
We now build a standard iteration scheme defined by
\begin{equation}\label{iterationsch}
 \left\{
 \begin{array}{l}
 y_{m+1}(t)=t+{\ds\int_{0}^{t}}(t-s)V(s)y_{m}(s) \, \mbox{ds},\eqskip
 y_{0}(t)=t
 \end{array}
 \right.
\end{equation}
which converges to the solution of equation~\eqref{controlproblem} -- this is a classical result from the theory of ordinary differential
equations which may be found, for instance, in~\cite{CL}, and which also follows from the computations below.

We shall now prove by induction that
\begin{equation}
 \label{indhyp}
 |y_{m}(t)-y_{m-1}(t)|\leq \frac{\ds t^{m+1}}{\ds (m+1)^{m+1}}\left[{\ds\int_{0}^{t}}|V(s)|\, \mbox{ds}\right]^{m}.
\end{equation}
From~\eqref{iterationsch} it follows that
\[
 y_{1}(t)-y_{0}(t)={\ds\int_{0}^{t}} (t-s) V(s) y_{0}(s) \, \mbox{ds}
\]
and thus
\[
 |y_{1}(t)-y_{0}(t)|\leq {\ds\int_{0}^{t}}s ( t-s) |V(s)|  \, \mbox{ds}.
\]
If we define the sequence of functions $f_{m}$ by
\[
 f_{m}(s) = (t-s)\frac{\ds s^{m+1}}{\ds (m+1)^{m+1}},
\]
the above may be written as
\[
 |y_{1}(t)-y_{0}(t)|\leq {\ds\int_{0}^{t}} f_{0}(s) |V(s)|  \, \mbox{ds}.
\]
Since $f_{0}(s) = s(t-s)\leq t^2/4$, we have that
\[
 |y_{1}(t)-y_{0}(t)|\leq \frac{\ds t^{2}}{\ds 4}{\ds\int_{0}^{t}} |V(s)|  \, \mbox{ds},
\]
and thus the induction hypothesis~\eqref{indhyp} holds for $m=1$.
Assume now that~\eqref{indhyp} holds.
It follows from~\eqref{iterationsch} that
\[
\begin{array}{lll}
 |y_{m+1}(t)-y_{m}(t)| & = & \left|{\ds\int_{0}^{t}} (t-s) V(s)\left[y_{m}(s)-y_{m-1}(s)\right]\, \mbox{ds}\right|\eqskip
 & \leq & {\ds\int_{0}^{t}} (t-s)\cdot |V(s)|\cdot \left|y_{m}(s)-y_{m-1}(s)\right|\, \mbox{ds}
\end{array}
 \]
 and, using~\eqref{indhyp}, we obtain
 \[
 \begin{array}{lll}
  |y_{m+1}(t)-y_{m}(t)| & \leq & {\ds\int_{0}^{t}} (t-s) |V(s)| \frac{\ds s^{m+1}}{\ds (m+1)^{m+1}}
  \left[{\ds\int_{0}^{s}}|V(r)|\, \mbox{dr}\right]^{m}  \, \mbox{ds}\eqskip
  & = & {\ds\int_{0}^{t}}  f_{m}(s)|V(s)| \left[{\ds\int_{0}^{s}}|V(r)|\, \mbox{dr}\right]^{m}  \, \mbox{ds}.
 \end{array}
 \]
Differentiating $f_{m}(s)$ with respect to $s$ and equating to zero, we obtain that this is maximal for $s=(m+1)t/(m+2)$, yielding
\[
 f_{m}(s)\leq \frac{\ds t^{m+2}}{\ds (m+2)^{m+2}}
\]
and so
\[
 |y_{m+1}(t)-y_{m}(t)|\leq \frac{\ds t^{m+2}}{\ds (m+2)^{m+2}}\left[{\ds\int_{0}^{t}}|V(s)|\, \mbox{ds}\right]^{m+1}
\]
as desired. Hence
\[
 \dsum_{m=0}^{\infty} |y_{m+1}(t)-y_{m}(t)|\leq \dsum_{m=0}^{\infty} \frac{\ds t^{m+2}}{\ds (m+2)^{m+2}}\left[{\ds\int_{0}^{t}} |V(s)|\, \mbox{ds} \right]^{m+1}
\]
On the other hand,
\[
 \dsum_{m=0}^{\infty} |y_{m+1}(t)-y_{m}(t)|=-y_{0} + \lim_{m\to\infty} y_{m}(t) = y(t)-t,
\]
yielding
\[
 |y(t)-t| \leq \dsum_{m=0}^{\infty} \frac{\ds t^{m+2}}{\ds (m+2)^{m+2}} \left[{\ds\int_{0}^{t}} |V(s)|\, \mbox{ds} \right]^{m+1}.
\]
Taking $t$ to be one finishes the proof.
\end{proof}

We shall finally present a proof of the fact that in order to maximize the determinant it is sufficient to consider non-negative
potentials. This is based on a comparison result for linear second order ordinary differential equations which we believe to
be interesting in its own right, but which we could not find in the literature.

\begin{prop}
 \label{positivitythm}
 Assume $V$ is in $L^\balpha[0,T]$, and let $u$ and $v$ be the solutions of the initial value problems defined by
 \[
\begin{array}{lll}
  \left\{
 \begin{array}{l}
 -u''(t)+|V(t)|u(t)=0\eqskip
 u(0)=0, \;\; u'(0)=1
 \end{array}
 \right.
& and &
 \left\{
 \begin{array}{l}
 -v''(t)+V(t)v(t)=0\eqskip
 v(0)=0, \;\; v'(0)=1
 \end{array}
 \right.
\end{array}.
\]
Then $u(t)\geq v(t)$ for all $0\leq t\leq T$.
\end{prop}
\begin{proof}
 The proof is divided into two parts. We first show that if the potential $V$ does not remain essentially non-negative, then $u$ must
 become larger than $v$ at some point, while never being smaller for smaller times. We then prove that once $u$ is strictly larger
 than
 $v$ for some time $t_{1}$, then it must remain larger for all $t$ greater than $t_{1}$.
 
 We first note that solutions of the above problems are at least in $AC^2[0,T]$, and
 thus continuously differentiable on $[0,T]$.
 Furthermore, $u$ is always positive,
 since $u(0)=0$, $u'(0)=1$ and $u''(0)\geq0$.
 Let now $w=u-v$. While $V$ remains essentially non-negative, $v$ also remains positive and $w$ vanishes identically.
 Assume now that there exists a time $t_{0}$ such that for $t$ in $(0,t_{0})$ the potential $V$ is essentially non-negative and on arbitrarily small
 positive neighbourhoods of $t_{0}$ $V$ takes on negative values on sets of positive measure. Then $w''(t)=|V(t)|w(t)+\left[ |V(t)|-V(t)\right]v$.
 close to $t_{0}$ (and zero elsewhere on these small neighbourhoods of $t_{0}$).
 Since $w(t_{0})=w'(t_{0})=0$ and $w''$ is non-negative and strictly positive on sets of positive measure contained in these neighbourhoods, $w$
 will take on positive values on arbitrarily small positive neighbourhoods of $t_{0}$.
 Define now $z(t) = u^{2}(t)-v^{2}(t)$. Then $z'(t) = u(t)u'(t)-v(t)v'(t)$ and
 \[
  \begin{array}{lll}
   z''(t) & = & \left[u'(t)\right]^2+u(t)u''(t)-\left[v'(t)\right]^2-v(t)v''(t)\eqskip
   & = & \left[u'(t)\right]^2-\left[v'(t)\right]^2 + |V(t)|z(t) + \left[|V(t)|-V(t)\right]v^2.
  \end{array}
 \]
 From the previous discussion above, we may assume the existence of a positive value $t_{1}$ such that both $z$ and $z'$ are positive at $t_{1}$ (and thus
 in a small positive neighbourhood $t_{1},t_{1}+\delta))$, while $z$ is non-negative for all $t$ in $(0,t_{1})$. Letting
 \[
  a(t) = \frac{\ds \left[u'(t)\right]^2-\left[v'(t)\right]^2}{\ds u(t)u'(t)-v(t)v'(t)},
 \]
which is well-defined and bounded on $[t_{1},t_{1}+\delta)$, we may thus write
 \[
  \begin{array}{lll}
   z''(t) & = & a(t)z'(t)+|V(t)|z(t) + \left[|V(t)|-V(t)\right]v^2 \eqskip
   & \geq & a(t) z'(t),
   \end{array}
 \]
 for $t$ in $[t_{1},t_{1}+\delta)$. Then $z''(t)-a(t)z(t)\geq0$ and upon multiplication by
 \[
 e^{-{\ds\int_{t_{1}}^{t}} a(s)ds}
 \]
 and integration between $t_{1}$ and $t$ we obtain
 \[
  z'(t)\geq e^{{\ds\int_{t_{1}}^{t}} a(s)ds}z'(t_{1}).
 \]
This yields the positivity of $z'(t)$ (and thus of $z(t)$) for $t$ greater than $t_{1}$. Combining this with the first part of the proof shows that
$z$ is never negative for positive $t$.
\end{proof}

\noi Henceforth, we restrict the search for maximizing potentials to non-negative functions.
Besides, it is clear from the proof that the result may be generalised to the comparison
of two potentials $V_{1}$ and $V_{2}$ where $V_{1}\geq V_{2}$, provided $V_{1}$ is non-negative.


\section{Maximization of the determinant over $L^1$ potentials\label{s3}} 
\noi By virtue of the analysis of the previous sections, problem (\ref{prob}) for
$\balpha=1$ can be recast as the following optimal control problem:
\begin{equation} \label{eq0}
  y(1) \to \max,
\end{equation}
\begin{equation} \label{eq1}
-\ddot{y}+Vy=0,\quad y(0)=0,\quad \dot{y}(0)=1,
\end{equation}
over all non-negative potentials $V \geq 0$ in $L^1[0,1]$ such that
\begin{equation} \label{eq1a}
 \int_0^1 V(t)\,\d t \leq A
\end{equation}
for fixed positive $A$.
In order to prove Theorem~\ref{thm1},
a family of auxiliary problems is introduced: in addition to (\ref{eq1}), the
potentials are assumed essentially bounded and such that
\begin{equation} \label{eq2}
  0 \leq V(t) \leq B,\quad \text{a.a.\ }t \in [0,1],
\end{equation}
for a fixed positive $B$. To avoid the trivial solution $V \equiv B$, we
suppose $B>A$ and henceforth study the properties of problem (\ref{eq0}-\ref{eq2}).

Setting $x:=(y,\dot{y},x_3)$, the auxiliary problem can be rewritten $-x_1(1) \to \min$
under the dynamical constraints
\begin{eqnarray} \label{eq4a}
  \dot{x}_1 &=& x_2,\\
  \dot{x}_2 &=& Vx_1,\\
  \dot{x}_3 &=& V, \label{eq4c}
\end{eqnarray}
$V$ mesurable valued in $[0,B]$,
and the boundary conditions $x(0)=(0,1,0)$, free $x_1(1)$ and $x_2(1)$, $x_3(1) \leq A$.

\begin{prop} Every auxiliary problem has a solution.
\end{prop}
\begin{proof} The set of admissible controls is obviously nonempty, the control is
valued in a fixed compact set, and the field of velocities
\[ \{ (x_2,Vx_1,V),\quad V \in [0,B] \} \subset \R^3 \]
is convex for any $x \in \R^3$; according to Filippov Theorem \cite{agrachev-2004a},
existence holds.\end{proof}

\noi Let $V$ be a maximizing potential for (\ref{eq0}-\ref{eq2}),
and let $x=(y,\dot{y},x_3)$ be the associated trajectory.
According to Pontrjagin maximum principle \cite{agrachev-2004a},
there exists a nontrivial pair $(p^0,p) \neq (0,0)$, $p^0 \leq 0$ a constant and
$p:[0,1] \to (\R^3)^*$ a Lip\-schitz covector function such that, a.e.\ on $[0,1]$,
\begin{equation} \label{eq5}
  \dot{x}=\frp{H}{p}(x,V,p),\quad \dot{p}=-\frp{H}{x}(x,V,p),
\end{equation}
and
\begin{equation} \label{eq6}
  H(x(t),V(t),p(t)) = \max_{v \in [0,B]} H(x(t),v,p(t))
\end{equation}
where $H$ is the Hamiltonian function
\begin{eqnarray*}
  H(x,V,p) &:=& pf(x,V)\\
  &=& p_1x_2 + (p_2x_1+p_3)V.
\end{eqnarray*}
(There $f(x,V)$ denotes the dynamics (\ref{eq4a}-\ref{eq4c}) in compact form.)
Moreover, in addition to the boundary conditions on $x$, the following transversality
conditions hold: $p_1(1)=-p^0$, $p_2(1)=0$, and $p_3(1) \leq 0$ with complementarity
\[ (x_3(1)-A)\,p_3(1) = 0. \]
As is clear from (\ref{eq5}), $p_3$ is constant, $\dot{p}_1=-p_2$ and
\begin{equation} \label{eq7}
  -\ddot{p}_2+Vp_2 = 0.
\end{equation}
The dynamical system (\ref{eq4a}-\ref{eq4c}) is bilinear in $x$ and $V$,
\[ \dot{x} = F_0(x)+VF_1(x),\quad
   F_0(x) = \left[ \begin{array}{c} x_2\\   0\\ 0 \end{array} \right],\quad
   F_1(x) = \left[ \begin{array}{c}   0\\ x_1\\ 1 \end{array} \right], \]
so $H=H_0+VH_1$ with $H_0(x,p)=pF_0(x)$, $H_1(x,p)=pF_1(x)$.
Let $\Phi(t):=H_1(x(t),p(t))$ be the evaluation of $H_1$ along the extremal $(x,V,p)$;
it is a Lip\-schitz function, and
the maximization condition (\ref{eq6}) implies that $V(t)=0$ when $\Phi(t)<0$,
$V(t)=B$ when $\Phi(t)>0$. If $\Phi$ vanishes identically on some interval of nonempty
interior, the control is not directly determined by (\ref{eq6}) and is termed
\emph{singular}. Subarcs of the trajectory corresponding to $V=0$, $V=B$ and singular control,
are labelled $\gamma_0$, $\gamma_+$ and $\gamma_s$, respectively. Differentiating
once, one obtains $\dot{\Phi}(t)=H_{01}(x(t),p(t))$,
where $H_{01}=\{H_0,H_1\}$ is the Poisson
bracket of $H_0$ and $H_1$; $\Phi$ is so $\W^{2,\iy}$ and, differentiating again,
\begin{equation} \label{eq8}
  \ddot{\Phi} = H_{001}+V H_{101}
\end{equation}
with length three brackets $H_{001}=\{H_0,H_{01}\}$, $H_{101}=\{H_1,H_{01}\}$.
Computing,
\[ H_{01}=-p_1x_1+p_2x_2,\quad H_{001}=-2p_1x_2,\quad H_{101}= 2p_2x_1. \]
In particular, using the definition of $H$,
\begin{equation} \label{eq9}
  \ddot{\Phi}-4V\Phi =-2(H+p_3 V),\quad \text{a.a.\ }t \in [0,1]
\end{equation}
where $H$ and $p_3$ are constant along any extremal.

There are two possibilities for
extremals depending on whether $\dot{p}_2(1)=p^0$ is zero or not (so-called abnormal
or normal cases).

\begin{lemma} \label{prop0}
The cost multiplier $p^0$ is negative and one can set $p'_2(1)=-1$.
\end{lemma}
\begin{proof} Suppose by contradiction that $p^0=0$. Then $p_2(1)=\dot{p}_2(1)=0$, so
(\ref{eq7}) implies that $p_2$ (and $p_1$) are identically zero. Since $(p^0,p) \neq
(0,0)$, $p_3$ must be negative so $\Phi=p_2x_1+p_3=p_3<0$ and $V$ is identically zero
on $[0,1]$, which is impossible (the zero control is admissible but readily not
optimal). Hence $p^0$ is negative, and one can choose $p^0=-1$ by homogeneity in
$(p^0,p)$.\end{proof}

\begin{lemma} \label{prop2}
The constraint $\int_0^1 V\,\d t \leq A$ is strongly active ($p_3<0$).
\end{lemma}
\begin{proof} Assume by contradiction that $p_3=0$.
Since $-x_1''+Vx_1=0$ with $x_1(0)=0$,
$x_1'(0)=1$ and $V \geq 0$, $x_1$ is non-negative on $[0,1]$ (see
Proposition~\ref{prop10}).
Integrating, one has $x_1(t) \geq t$ on $[0,1]$.
Symmetrically, $-p_2''+Vp_2=0$ with $p_2(1)=0$, $p_2'(1)=-1$, so
one gets that $p_2(t) \geq 1-t$ on $[0,1]$. Now, $p_3=0$ implies
$\Phi=p_2x_1 \geq t(1-t)$, so $\Phi>0$ on $(0,1)$: $V=B$ a.e., which
is impossible since $B>A$.\end{proof}
\noi As a result, since $\Phi(0)=\Phi(1)=p_3$, $\Phi$ is negative in the neighbourhood
of $t=0+$ and $t=1-$, so an optimal trajectory starts and terminates with $\gamma_0$ arcs.

\begin{lemma} \label{lemma1}
There is no interior $\gamma_0$ arc.
\end{lemma}
\begin{proof} If such an interior arc existed, there would exist $\bar{t} \in (0,1)$
such that $\Phi(\bar{t})=0$ and $\dot{\Phi}(\bar{t}) \leq 0$ ($\dot{\Phi}(\bar{t})>0$
would imply $\Phi>0$ in the neighbourhood of $\bar{t}+$, contradicting $V=0$); but
then $\ddot{\Phi}=-2H<0$ for $t \geq \bar{t}$ would result in $\Phi<0$ for $t >
\bar{t}$, preventing $\Phi$ for vanishing again before $t=1$.\end{proof}

\begin{lemma} \label{lemma2}
If $H+p_3 B<0$, there is no interior $\gamma_+$ arc.
\end{lemma}
\begin{proof} By contradiction again: there would exist $\bar{t} \in (0,1)$
such that $\Phi(\bar{t})=0$ and $\dot{\Phi}(\bar{t}) \geq 0$ ($\dot{\Phi}(\bar{t})<0$
would imply $\Phi<0$ in the neighbourhood of $\bar{t}+$, contradicting $V=B$);
along $\gamma_+$,
\[ \ddot{\Phi}-4B\Phi =-2(H+p_3 B)>0, \]
so
\begin{equation} \label{eq11}
  \Phi = \frac{H+p_3 B}{2B}(1-\ch(2\sqrt{B}(t-\bar{t})))
        + \frac{\dot{\Phi}(\bar{t})}{2\sqrt{B}}\,\sh(2\sqrt{B}(t-\bar{t})),
\end{equation}
and $\Phi>0$ if $t > \bar{t}$, preventing $\Phi$ for vanishing again before
$t=1$.\end{proof}

\noi Along a singular arc, $\Phi \equiv 0$ so (\ref{eq8}) allows to determine the
singular control provided $H_{101} \neq 0$ ("order one" singular).
A necessary condition for optimality
(the Legendre-Clebsch condition) is that $H_{101} \geq 0$ along such an arc; order one
singular arcs such that $H_{101}>0$ are called \emph{hyperbolic} \cite{bonnard-2003a}.

\begin{lemma} Singulars are of order one and hyperbolic; the singular control is
constant and equal to $-H/p_3$.
\end{lemma}
\begin{proof} $\Phi=p_2x_1+p_3 \equiv 0$ implies $H_{101}=2p_2x_1=-2p_3>0$, so any
singular is of order one and hyperbolic; (\ref{eq9}) then tells $H+p_3 V=0$, hence the
expression of the singular control.\end{proof}

\begin{prop} A maximizing potential is piecewise constant.
\end{prop}
\begin{proof} According to the maximum principle, the trajectory associated with an
optimal potential is the concatenation of (possibly infinitely many) $\gamma_0$,
$\gamma_+$ and $\gamma_s$ subarcs. By Lemma~\ref{lemma1}, there are exactly two
$\gamma_0$ arcs; if $V$ has an infinite number of discontinuities, it is necessarily
due to the presence of infinitely many switchings between $\gamma_+$ and $\gamma_s$
subarcs with $H+p_3 B<0$ (this quantity must be nonpositive to ensure admissibility of
the singular potential, $V=-H/p_3 \leq B$, and even negative since otherwise
$\gamma_+$ and $\gamma_s$ would be identical, generating no discontinuity at all). By
Lemma~\ref{lemma2}, there is no $\gamma_+$ subarc if $H+p_3 B<0$. There are thus only
finitely many switchings of the potential between the constant values $0$, $B$ and
$-H/p_3$. 
\end{proof}

\begin{cor} \label{cor1}
An optimal trajectory is either of the form $\gamma_0\gamma_+\gamma_0$,
or $\gamma_0\gamma_s\gamma_0$.
\end{cor}
\begin{proof} There is necessarily some minimum $\bar{t} \in (0,1)$ such that
$\Phi(\bar{t})=0$ (otherwise $V \equiv 0$ which would contradict the existence of
solution). There are two cases, depending on the order of the contact of the extremal
with $H_1=0$.\\

\noi (i) $\dot{\Phi}(\bar{t}) \neq 0$ ("regular switch" case). Having started with
$\Phi(0)=p_3<0$, necessarily $\dot{\Phi}(\bar{t})>0$ and there is a switch from $V=0$
to $V=B$ at $\bar{t}$. According to Lemma~\ref{lemma2}, this is only possible if
$H+p_3 B \geq 0$. Along $\gamma+$, $\Phi$ is given by (\ref{eq11}) and vanishes again
at some $\tilde{t} \in (\bar{t},1)$ as the trajectory terminates with a $\gamma_0$
subarc. Since $H+p_3 B \geq 0$ and $\dot{\Phi}(\bar{t})<0$, $\dot{\Phi}(\tilde{t})$ is
necessarily negative, so $\Phi<0$ for $t > \tilde{t}$ and the structure is
$\gamma_0\gamma_+\gamma_0$.\\

\noi (ii) $\dot{\Phi}(\bar{t})=0$. The potential being piecewise continuous,
$\ddot{\Phi}$ has left and right limits at $\bar{t}$ (see (\ref{eq8})). Having started
with $V=0$, $\ddot{\Phi}(\bar{t}-)=-2H$ is negative; assume the contact is of order
two, that is $\ddot{\Phi}(\bar{t}+) \neq 0$ ("fold" case, \cite{bonnard-2003a}).
Clearly, $\ddot{\Phi}(\bar{t}+)<0$ is impossible as it would imply $\Phi<0$ for $t >
\bar{t}$ (and hence $V \equiv 0$ on $[0,1]$). So $\ddot{\Phi}(\bar{t}+)>0$
("hyperbolic fold", \emph{ibid}), and there
is a switching from $V=0$ to $V=B$. Along $\gamma_+$,
\[ \Phi = \frac{H+p_3 B}{2B}(1-\ch(2\sqrt{B}(t-\bar{t}))) \]
by (\ref{eq11}) again, and $H+p_3 B$ must be negative for $\ddot{\Phi}(\bar{t}+)>0$ to
hold. Then $\Phi>0$ for $t > \bar{t}$, which contradicts the termination by a
$\gamma_0$ subarc. So $\ddot{\Phi}(\bar{t}+)=0$, that is $V$ jumps from $0$ to singular,
$V=-H/p_3$, at $\bar{t}$. As it
must be admissible, $H+p_3 B \leq 0$. If $H+p_3 B=0$, the singular control is
saturating, $V=B$, and $\gamma_+$ and $\gamma_s$ arcs are identical; there cannot be
interior $\gamma_0$ (Lemma~\ref{lemma1}), so the structure is
$\gamma_0\gamma_s\gamma_0$ ($=\gamma_0\gamma_+\gamma_0$, in this case).
Otherwise, $H+p_3 B<0$ and
Lemma~\ref{lemma2} asserts that $\gamma_s\gamma_+$ connections are not possible:
The structure is also $\gamma_0\gamma_s\gamma_0$.\end{proof}

\begin{proof}[Proof of Theorem~\ref{thm1}] The proof is done in three steps: first,
existence and uniqueness are obtained for each auxiliary problem (\ref{eq1}-\ref{eq2})
with bound $V \leq B$ on the potential, $B$ large enough;
then existence for the original problem
(\ref{eq0}-\ref{eq2}) with unbounded control is proved. Finally, uniqueness is obtained.\\

\noi (i) As a result of Corollary~\ref{cor1}, each auxiliary problem can be reduced
to the following finite dimensional question: given $B>A>0$, maximize
$y(1)=y(s,\ell)$ w.r.t.\ $s \in [0,1]$, $\ell \leq 2\min\{s,1-s\}$ and $\ell
\geq A/B$, where $y(s,\ell)$
is the value at $t=1$ of the solution of (\ref{eq1}) generated by the
potential equal to the characteristic function of the interval $[s-\ell/2,s+\ell/2]$
times $A/\ell$ (the constraint (\ref{eq1a}) is active by virtue
of Lemma~\ref{prop2}). Computing,
\[ y(s,\ell) = (1-\ell)\,\ch\,\sqrt{A\ell}
   +A[(s-s^2)+(1/A-1/2)\ell+\ell^2/4]\,\shc\,\sqrt{A\ell} \]
where $\shc(z):=(1/z)\,\sh\,z$.
The function to be maximized is continuous
on the compact triangle defining the constraints on $(s,\ell)$,
so one retrieves existence. As we know that an optimal arc must start and end with
$\gamma_0$ arcs, the solution cannot belong to the part of the boundary corresponding
to $\ell=2\min\{s,1-s\}$, extremities included; as a result, $\partial y/\partial s$
must be zero at a solution, so $s=1/2$ (potential symmetric w.r.t.\ $t=1/2$). Since
one has $y(s=1/2,\ell)=(1+A/4)+A^2\ell/24+o(\ell)$, too small a $\ell$ cannot
maximize the function; so, for $B$ large enough, the point on the boundary
$\ell=A/B$ cannot be solution, and one also has that
\[ \frp{y}{\ell}(s,\ell)|_{s=1/2} =
   \frac{(A(\ell-1)^2-4\ell)(\ch\,\sqrt{A\ell}-\shc\,\sqrt{A\ell})}{8\ell} \]
must vanish. As $\ell<1$, one gets
$\ell=\ell(A):=A/(1+\sqrt{1+A})^2$, hence the expected value for the maximum determinant.\\

\noi (ii) The mapping $V \mapsto -y(1)$ is continuous on $L^1[0,1]$ by
Proposition~\ref{prop10}.
Let $C:=\{V \in L^1[0,1]\,|\,\|V\|_1 \leq A,\ V \geq 0 \text{ a.a.}\}$; it is closed
and nonempty.
For $k \in \NM$, $k>A$ ($A>0$ fixed),
consider the sequence
of auxiliary problems with essential bounds $V \leq k$.
For $k$ large enough, the solution does not depend on $k$ according to point (i),
hence the stationarity of the sequence of solutions to the auxiliary problems.
The sequence of subsets $C \cap \{V \in L^\iy[0,1]\,|\,\|V\|_\iy \leq B_k\}$
is increasing and dense in $C$, so the lemma below ensures existence for the
original problem.

\begin{lemma} \label{lemma4}
Let $f:E \to \R$ be continuous on a normed space $E$, and consider
\[ f(x) \to \min,\quad x \in C \text{ nonempty closed subset of }E. \]
Assume that there exists an increasing sequence $(C_k)_k$
of subsets of $C$ such that (i) $\cup_k C_k$ is dense in $C$,
(ii) for each $k \in \NM$, there is a minimizer $x_k$ of $f$ in $C_k$.
Then, if $(x_k)_k$ is stationary, $\lim_{k \to \iy} x_k$ is a minimizer of $f$ on $E$.
\end{lemma}
\begin{proof} Let $\bar{x}$ be the limit of the stationary sequence $(x_k)_k$; assume,
by contradiction, that $\bar{f}:=f(\bar{x}) > \inf_C f$ (including the case $\inf_C
f=-\iy$). Then there is $y \in C$
such that $f(y)<\bar{f}$. By density, there is $k \in \NM$ and $y_k \in C_k$
such that $\|y-y_k\| \leq \veps$ where, by continuity of $f$ at $y$,
$\veps>0$ is choosen such
that $\|x-y\| \leq \veps \implies |f(x)-f(y)| \leq d/2$ ($d:=\bar{f}-f(y)>0$);
then $f(y_k) \leq f(y)+d/2 < \bar{f}$, which is contradictory as $k$ can be taken
large enough in order that $x_k=\bar{x}$. So $f(\bar{x})=\inf_C f$, whence existence.
\end{proof}

\noindent \emph{(End of proof of Theorem~\ref{thm1}.)}
(iii) the Pontrjagin maximum principle can be applied to (\ref{eq0}-\ref{eq1a})
with an optimal potential $V$ in $L^1[0,1]$ since the function $f(x,V)$
defining the dynamics
(\ref{eq4a}-\ref{eq4c}) has a partial derivative uniformly (w.r.t.\ $x$) dominated by
an integrable function (see \cite[\S~4.2.C, Remark~5]{cesari-1983a}):
\[ \left|\frp{f}{x}(x,V(t))\right| \leq 1+|V(t)|,\quad \text{a.a.\ }t \in [0,1]. \]
The Hamiltonian is unchanged, but the constraint on the potential is now just $V \geq 0$;
as a consequence, $\Phi=p_2x_1+p_3$ must be nonpositive because of the maximization
condition. Normality is proved as in Lemma~\ref{prop0} and $p_2(1)$ can be set to $-1$.
Regarding strong activation of the $L^1$ constraint ($p_3<0$) the argument at the
begining of the proof of Lemma~\ref{prop2} immediately implies that $p_3$ must be
negative as $\Phi$ cannot be positive, now. So $\Phi(0)=\Phi(1)=p_3<0$ and $V$ is
equal to $0$ in the neighbourhood of $t=0+$ and $t=1-$. Because existence holds,
$\Phi$ must first vanish at some $\bar{t} \in (0,1)$ (otherwise $V \equiv 0$ which is
obviously not optimal); necessarily, $\dot{\Phi}(\bar{t})=0$. One verifies as before
that there cannot be interior $\gamma_0$ subarcs, which prevents accumulation of
$\gamma_0\gamma_s$ or $\gamma_s\gamma_0$ switchings. In particular, $V$ must be piecewise
constant and the right limit $\ddot{\Phi}(\bar{t}+)$ exists. The same kind of reasoning
as in the proof of Corollary~\ref{cor1} rules out the fold case $\ddot{\Phi}(\bar{t}+)$
nonzero; so $\ddot{\Phi}(\bar{t}+)=0$ and $V$ switches from $0$ to singular at
$\bar{t}$. The impossibility of interior $\gamma_0$ subarcs implies a
$\gamma_0\gamma_s\gamma_0$ structure, which is the solution obtained before.
\end{proof}

\begin{remark}
This proves in particular that the generalized
"impulsive" potential equal to $A$ times the
Dirac mass at $t=1/2$ is not optimal, whatever $A>0$,
the combination $\gamma_0\gamma_s\gamma_0$ giving a better cost.
More precisely, the optimal value of the determinant is $1+A/4+A^3/192+o(A^3)$,
so it is asymptotic when $A$ tends to zero to the value obtained for the Dirac mass
(equal to $1+A/4$, as is clear from (i) in the proof before).
\end{remark}


\section{Maximization of the determinant over $L^q$ potentials, $q>1$ \label{s4}}
\noi We begin the section by proving that maximizing the determinant over $L^1$ potentials
is estimated by maximizing the determinant over $L^\balpha$, letting $\balpha$ tend to
one. To this end, we first establish the following existence result. (Note that the
proof is completely different from the existence proof in Theorem~\ref{thm1}.) 
Let $A>0$, and consider the following family of problems for $\balpha$ in $[1,\iy]$:
\begin{equation} \label{eq30}
  \DD_V = \frac{1}{2}\det(-\Delta+V) \to \max,\quad \|V\|_\balpha \leq A.
\end{equation}

\begin{prop} \label{prop8}
Existence holds for $\balpha \in [1,\iy]$.
\end{prop}

\begin{proof} The case $\balpha=\iy$ is obvious, while Theorem~\ref{thm1} deals with 
$\balpha=1$. Let then $\balpha$ belong to $(1,\iy)$. By using Gronwall lemma as in
(\ref{eqb3}), it is clear that $\DD$ is bounded on the closed ball of radius $A$ of
$L^\balpha[0,1]$. So the value of the problem is finite. Let $(V_k)_k$ be a maximizing
sequence. As $\|V_k\|_\balpha \leq A$ for any $k$, up to taking a subsequence one can
assume the sequence to be weakly-$*$ converging in $L^\balpha \simeq (L^\bbeta)^*$
($1/\balpha+1/\bbeta=1$) towards some $\bar{V}$. Clearly, $\|\bar{V}\|_\balpha \leq A$.
Let $x_k$ be associated with $V_k$ according to (\ref{eqb1}). The sequence $(x_k)_k$
so defined is bounded and equicontinuous as, for any $t \leq s$ in $[0,1]$,
\begin{eqnarray*}
  |x_k(t)-x_k(s)| & \leq & \int_t^s |C(V_k(\tau))|\cdot|x_k(\tau)|\,\d\tau \\
  & \leq & \|x_k\|_\infty \|1+|V_k|\|_\balpha |t-s|^{1/\bbeta}\\ 
  & \leq & e^{1+A} (1+A) |t-s|^{1/\bbeta}\\ 
\end{eqnarray*}
by H\"older inequality. Using Ascoli's Theorem (and taking a subsequence), $(x_k)_k$
converges uniformly towards some $\bar{x}$. As $\bar{x}(1)=\lim_k x_k(1)$ is equal to
the value of the problem, it suffices to check that $\bar{x}'=C(\bar{V})\bar{x}$ to
conclude. Being a bounded sequence, $(C(V_k))_k$ is equicontinuous in $L^\balpha \simeq
(L^\bbeta)^*$, and $(x_k \cdot \chi_{|[0,t]})_k$ converges towards
$\bar{x} \cdot \chi_{|[0,t]}$
in $L^\bbeta$ for any $t$ in $[0,1]$ ($\chi_{|[0,t]}$ denoting the characteristic
function of the interval $[0,t]$). So
\begin{eqnarray*}
  \bar{x}(t) &=& \bar{x}(0) + \lim_k \int_0^t C(V_k(\tau))x_k(\tau)\,\d\tau\\
             &=& \bar{x}(0) + \int_0^t C(\bar{V}(\tau))\bar{x}(\tau)\,\d\tau,
		    \quad t \in [0,1],
\end{eqnarray*}
which concludes the proof.\end{proof}

\begin{remark} The same proof also gives existence for the minimization problem in
$L^\balpha[0,1]$, $\balpha \in (1,\iy)$.
\end{remark}

\begin{prop} For a fixed $A>0$, the value function 
\[ v_\balpha := \sup_{\|V\|_\balpha \leq A} \DD_V,\quad \balpha \in [1,\iy], \]
is decreasing and $v_\balpha$ tends to $v_1$ when $\balpha \to 1+$.
\end{prop}

\begin{proof} Let $1 \leq \balpha \leq \bbeta \leq \iy$.
For $V$ in $L^\bbeta[0,1] \subset L^\balpha[0,1]$, $\|V\|_\balpha \leq \|V\|_\bbeta$ so
the radius $A$ ball of $L^\bbeta$ is included in the radius $A$ ball of $L^\balpha$.
As a result of the inclusion of the admissible potentials, $v_\balpha \geq v_\bbeta$.
Because of monotonicity, the limit $\bar{v}:=\lim_{\balpha \to 1+} v_\balpha$ exists,
and $\bar{v} \leq v_1$. Moreover, as proven in Theorem~\ref{thm1}, the unique
maximizing potential $V_1$ for $\balpha=1$ is actually essentially bounded so
\[ \DD(A\cdot V_1/\|V_1\|_\balpha) \leq v_\balpha \]
for any $\balpha \geq 1$. By continuity of $\DD$ on $L^1$, the left-hand side of the
previous inequality tends to $\DD_{V_1}=v_1$ when $\balpha \to 1+$, and one conversely
gets that $v_1 \leq \bar{v}$.\end{proof}

Fix $q > 1$ and $A>0$. As in Section~\ref{s3}, we set $x:=(y,\dot{y},x_3)$
to take into account the $L^q$ constraint.
Then problem (\ref{eq30}) can be rewritten
$-x_1(1) \to \min$ under the dynamical constraints
\begin{eqnarray} \label{eq30a}
  \dot{x}_1 &=& x_2,\\
  \dot{x}_2 &=& Vx_1,\\
  \dot{x}_3 &=& V^q, \label{eq30c}
\end{eqnarray}
and the boundary conditions $x(0)=(0,1,0)$, free $x_1(1)$ and $x_2(1)$,
$x_3(1) \leq A^q$.
The potential is mesurable and can be assumed non-negative
thanks to the comparison result from Proposition~\ref{positivitythm}.
For such an $L^q$ optimal potential $V$, the Pontrjagin maximum principle holds
for the same reason as in the proof of Theorem~\ref{thm1} (see step (iii)).
So there exists a nontrivial pair $(p^0,p) \neq (0,0)$, $p^0 \leq 0$ a constant and
$p:[0,1] \to (\R^3)^*$ a Lip\-schitz covector function such that, a.e.\ on $[0,1]$,
\[ \dot{x}=\frp{H}{p}(x,V,p),\quad \dot{p}=-\frp{H}{x}(x,V,p), \]
and
\[ H(x(t),V(t),p(t)) = \max_{v \geq 0} H(x(t),v,p(t)) \]
where the Hamiltonian $H$ is now equal to
\begin{eqnarray*}
  H(x,V,p) &:=& pf(x,V)\\
  &=& p_1 x_2 + p_2 x_1 V + p_3 V^q.
\end{eqnarray*}
(With $f(x,V)$ denoting the dynamics (\ref{eq30a}-\ref{eq30c}) in compact form.)
In addition to the boundary conditions on $x$, the following transversality
conditions hold (note that $p_3$ is again a constant):
$p_1(1)=-p^0$, $p_2(1)=0$, and $p_3 \leq 0$ with complementarity
\[ (x_3(1)-A^q)\,p_3 = 0. \]
Although the system is not bilinear in $(x,V)$ anymore, the adjoint equation
\begin{equation} \label{eq31}
  -p_2''+V p_2 = 0
\end{equation}
holds unchanged, and one proves normality and strong activation of the $L^q$
cons\-traint similarly to the case $q=1$.

\begin{lemma}
The cost multiplier $p^0$ is negative.
\end{lemma}
\begin{proof}
By contradiction: if $p^0=0$, one has $p_2$ (and $p_1$) identically zero by
(\ref{eq31}), so $p_3$ cannot also be zero and must be negative.
Then, $H=p_3 V^q$ and the
maximization condition implies $V=0$ a.e.\ since $V$ is non-negative. This is
contradictory as the zero control is admissible but clearly not optimal.
\end{proof}

\noi We will not set $p^0=-1$ but will use the fact that $p_3$ is also negative to use a
different normalization instead.

\begin{lemma} \label{lem10}
The constraint $\int_0^1 V^q\,\d t \leq A^q$ is strongly active ($p_3<0$).
\end{lemma}
\begin{proof} Observe that, as in Lemma~\ref{prop2}, $p_2x_1$ is positive on $(0,1)$.
Now, assume by contradiction that $p_3=0$: then $H=p_1x_2+p_2x_1V$, which would
prevent maximization of $H$ on a nonzero measure subset. 
\end{proof}

\noi Define $\Psi:=p_2x_1$. As we have just noticed, it is positive on $(0,1)$, and
$\Psi(0)=\Psi(1)=0$ because of the boundary and transversality conditions
($x_1(0)=0$ and $p_2(1)=0$, respectively).

\begin{prop} \label{prop20}
One has
\[ \Psi''-|\Psi|^\alpha+2H = 0,\quad \alpha=q/(q-1), \]
\[ \Psi(0)=0,\quad \Psi(1)=0, \]
and
\begin{equation} \label{eq90}
  V = \frac{q}{4q-2}\sqrt[q-1]{\Psi}.
\end{equation}
\end{prop}
\begin{proof}
Because $\Psi$ is non-negative, it is clear
from the maximization condition that
\[ V(t) = \sqrt[q-1]{\frac{\Psi(t)}{-qp_3}} \] 
for all $t \in (0,1)$. Since $\Psi$ is an absolutely continuous function, we can
differentiate once to get
\[ \dot{\Psi} = -p_1x_1+p_2x_2, \]
and iterate to obtain
\[ \ddot{\Psi} = 2(p_2x_1V-p_1x_2) = 2(V\Psi-p_1x_2). \]
Using the fact that the Hamiltonian $H$ is constant along an extremal, and
substituing $p_1x_2$ by $H-V\Psi-p_3V^q$ and $V$ by its expression, the
following second order differential equation is obtained for $\Psi$:
\[ \Psi''-\frac{B_q}{\sqrt[q-1]{-p_3}}|\Psi|^{q/(q-1)}+2H = 0 \]
with
\[ B_q = \frac{4}{q^{1/(q-1)}}-\frac{2}{q^{q/(q-1)}} \cdot \]
We can normalize $p_3$ in order that $-p_3=B_q^{q-1}$, which gives the desired
differential equation for $\Psi$, as well as the desired expression for $V$.
\end{proof}

\begin{cor} \label{cor13}
The function $\Psi$ (and so $V$) is symmetric \emph{wrt.} $t=1/2$, and
$\Psi'(0)=H-c(A,q)$ with
\[ c(A,q) = \frac{1}{2} \left( \frac{A(4q-2)}{q} \right)^q. \]
\end{cor}
\begin{proof}
As a result of the previous proposition, the quantity
\[ \frac{1}{2}\Psi'^2-\frac{1}{\alpha+1}\Psi|\Psi|^\alpha+2H\Psi \]
is constant. In particular, $\Psi(0)=\Psi(1)=0$ implies that $\Psi'^2(0)=\Psi'^2(1)$.
Now, $\Psi'(0)=p_2(0)>0$ and $\Psi'(1)=-x_1(1)<0$ (same estimates as in
Lemma~\ref{prop2}), so $\Psi'(1)=-\Psi'(0)$. Setting
$\hat{\Psi}(t):=\Psi(1-t)$, one then checks that both $\Psi$ and $\hat{\Psi}$ verify
the same differential equation, with the same initial conditions: $\hat{\Psi}=\Psi$
and symmetry holds.
Finally, since the $\L^q$ constraint is
active (Lemma~\ref{lem10}),
\[ A^q = \int_0^1 |V|^q\,\d t
       = \left( \frac{q}{4q-2} \right)^q \int_0^1 |\Psi|^\alpha\,\d t, \]
and one can replace $|\Psi|^\alpha$ by $\Psi''+2H$ to integrate and obtain
\[ \Psi'(1)-\Psi'(0)+2H = \left( \frac{A(4q-2)}{q} \right)^q. \]
Hence the conclusion using $\Psi'(0)-\Psi'(1)=2\Psi'(0)$.
\end{proof}

\noi According to what has just been proved, $t \mapsto (\Psi(t),\Psi'(t))$
parameterizes the curve $y^2=f(x)$ where $f$ (that depends on $H$, $A$ and $q$) is
\begin{equation} \label{eq61}
  f(x) = \frac{2}{\alpha+1}x|x|^\alpha-4Hx+(H-c(A,q))^2 \quad (\alpha=q/(q-1)).
\end{equation}
Since $\Psi'(0)=H-c(A,q)$ is positive, $H > c(A,q) > 0$ and $f$ has a local minimum
(\emph{resp.}\ maximum) at $x=\sqrt[\alpha]{2H}$ (\emph{resp.}\ $-\sqrt[\alpha]{2H}$).

\begin{lemma} \label{lem12}
On $(c(A,q),\infty)$, there exists a unique $H$, denoted $h(A,q)$,
such that $\sqrt[\alpha]{2H}$ is a double root of $f$.
\end{lemma}
\begin{proof}
Evaluating,
\[ f(\sqrt[\alpha]{2H}) = (H-c(A,q))^2-\frac{2\alpha}{\alpha+1}(2H)^{1+1/\alpha}
                        =: g(H). \]
As $g(c(A,q))<0$ and $g(H) \to \infty$ when $H \to \infty$ (note that $1+1/\alpha<2$),
$g$ has one zero in $(c(A,q),\infty)$, and only one in this interval as is clear
inspecting $g''$.
\end{proof}

\begin{prop} \label{prop11}
The value $H$ of the Hamiltonian must belong to the nonempty open interval
$(c(A,q),h(A,q))$.
\end{prop}
\begin{proof} As a result of the previous lemma,
for $H>c(A,q)$ there are three possibilities for the curve $y^2=f(x)$
parameterized by $(\Psi,\Psi')$
depending on whether (i) $H$ belongs to $(c(A,q),h(A,q))$, (ii) $H=h(A,q)$, (iii)
$H>h(A,q)$. As is clear from Figure~\ref{fig1}, given the boundary conditions
$\Psi(0)=\Psi(1)=0$, $\Psi'(0)>0$ and $\Psi'(1)<0$, case (iii) is excluded. Now, in
case (ii), the point $(\!\sqrt[\alpha]{2h(A,q)},0)$ is a saddle equilibrium point,
which prevents connexions between $(\Psi,\Psi')=(0,h(A,q)-c(A,q))$ and
$(0,c(A,q)-h(A,q))$.
\end{proof}

\begin{cor} The function $\psi$ (and so $V$) is strictly increasing on
$[0,1/2]$.
\end{cor}
\begin{proof} The derivative of $\psi$ is strictly positive for $t$ in $(0,1/2)$ as is
clear from Figure~\ref{fig1}, case (i). As a consequence,
$\psi$ is strictly increasing on $[0,1/2]$, and so is $V$ by virtue of (\ref{eq90}).
\end{proof}

\begin{figure}
\centering \includegraphics[width=1.0\textwidth]{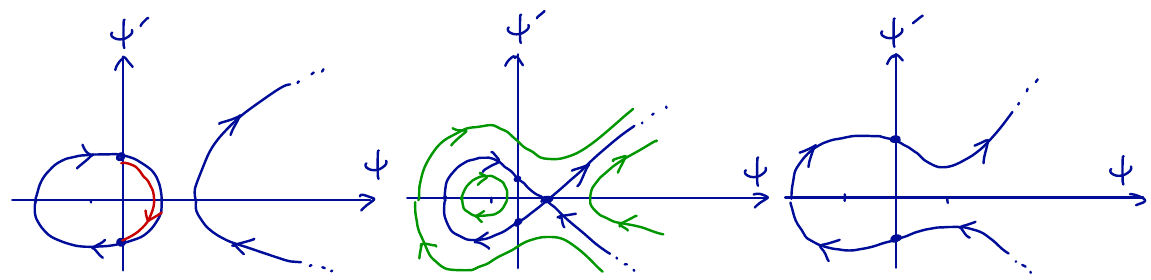}
\caption{Portraits of the curve $y^2=f(x)$ parameterized by $(\Psi,\Psi')$. From left
to right: case (i) $H \in (c(A,q),h(A,q))$ (the red part corresponds to $t \in [0,1]$),
(ii) $H=h(A,q)$ (phase portrait in green
for other boundary conditions close to the saddle equilibrium), (iii) $H>h(A,q)$.}
\label{fig1}
\end{figure}

\begin{proof}[Proof of Theorem~\ref{thm2}.] 
Let $q>1$ and $A>0$ be given.
Optimal potentials exist by
Proposition~\ref{prop8}.
Such an optimal potential must be given by $\Psi$ according to Proposition~\ref{prop20}.
By virtue of Proposition~\ref{prop11},
this function $\Psi$ is obtained as the solution $\Psi(\cdot,H)$ of
\begin{equation} \label{eq20}
  \Psi''-|\Psi|^\alpha+2H = 0,\quad \Psi(0)=0,\quad \Psi'(0)=H-c(A,q),
\end{equation}
for some $H$ in $(c(A,q),h(A,q))$ such that $\Psi(1,H)=0$. For any $H$ in this
interval, let us first notice that $(\Psi,\Psi')$ define a parameterization of
the bounded component of the curve $y^2=f(x)$ (see Figure~\ref{fig1}). Accordingly,
both $\Psi$ and $\Psi'$ are bounded, and the solution $\Psi(\cdot,H)$ of (\ref{eq20})
is defined globally, for all $t \in \R$. Hence, the function $H \mapsto \Psi(1,H)$ is
well defined on $(c(A,q),h(A,q))$. Proving that this mapping is injective will entail
uniqueness of an $H$ such that $\Psi(1,H)=0$, and thus uniqueness of the optimal
potential for the given $q>1$ and positive $\L^q$ bound $A$. Now, this mapping is
differentiable, and $(\frpp{\Psi}{H})(1,H)=\Phi(1)$ where $\Phi$ is the solution of
the following linearized differential equation (note that $\alpha=q/(q-1)>1$):
\[ \Phi''-\alpha\Psi^{\alpha-1}\Phi+2 = 0,\quad \Phi(0)=0,\quad \Phi'(0)=1. \]
The function $\Phi$ is non-negative in the neighbourhood of $t=0+$.
Let us denote $\tau \in
(0,\infty]$ the first possible zero of $\Phi$, and $\tau':=\min\{\tau,1\}$ (remember
that $\Psi>0$ on $(0,1)$). On $(0,\tau')$,
\[ \Phi'' = \alpha\Psi^{\alpha-1}\Phi-2 > -2, \]
so $\Phi > t(1-t)$ on $(0,\tau']$ by integration: necessarily, $\tau>1$. Then
$\Phi(1)>0$, so the mapping $H \mapsto \Psi(1,H)$ is strictly increasing on
$(c(A,q),h(A,q))$ and uniqueness is proved.
Regarding the regularity of the optimal potential, it is clear that
$\Psi$ is smooth on $(0,1)$. Besides, $\dot{\Psi}(0)$ is
positive and it suffices to write, for small enough $t>0$,
\[ \frac{\Psi^{\frac{1}{q-1}}(t)-0}{t-0}
 = t^\frac{2-q}{q-1} \left( \frac{\Psi(t)}{t} \right)^{\frac{1}{q-1}} \]
to evaluate the limit when $t \to 0+$ and obtain the desired conclusion for the
tangencies. (Note the bifurcation at $q=2$.) Same proof when $t \to 1-$.
\end{proof} 

In the particular case $q=2$, one has $\alpha=2$ and $t \mapsto (\psi(t),\psi'(t))$
parameterizes the elliptic curve (compare with~(\ref{eq61}))
\[ y^2 = \frac{2}{3}x^3-4Hx+(H-c(2,A))^2. \]
We know that this elliptic curve is not degenerate for $H$ in
$(c(2,A),h(2,A))$. The value $c(2,A)=9A^2/2$ is explicit (Corollary~\ref{cor13}),
while $h^*(A):=h(2,A)$ is implicitly defined Lemma~\ref{lem12}.
Using the birational change of variables $u=x$, $v=y\sqrt{6}$, the elliptic curve can
be put in Weierstra\ss\ form, $v^2 = 4u^3-g_2 v-g_3$, with
\begin{equation} \label{eq62}
g_2 = 24 H,\quad g_3 = -6(H-9A^2/2)^2.
\end{equation}
For $H$ in $(c(2,A),h(2,A))$, the real curve has two connected components in the plane
and is 
parameterized by $z \mapsto (\wp(z),\wp'(z))$, where $\wp$ is the Weierstra\ss\
elliptic function associated to the invariants (\ref{eq62}). Since $g_2$ and $g_3$ are
real, and since the curve has two components, the lattice $2\omega\Z+2\omega'\Z$
of periods of $\wp$ is rectangular: $\omega$ is real, $\omega'$ is purely imaginary,
and the bounded component of the curve is obtained for $z \in \R+\omega'$. The curve
degenerates for $H=h^*(A)$, so $h^*(A)$ can also be retrieved as the unique root in
$(9A^2/2,\iy)$ of the discriminant
\[ \Delta = g_2^3-27g_3^2 = 3\cdot 6^2(128 H^3-9(H-9A^2/2)^4) \]
of the cubic. We look for a time parameterization $z(t)$ such that
$\wp(z(t))=\psi(t)$ (since $u=x$), and $\wp'(z(t))=\psi'(t)\sqrt{6}$ (since
$v=y\sqrt{6}$).

\begin{lemma} $\displaystyle z(t) = \frac{2t-1}{2\sqrt{6}}+\omega'$
\end{lemma}

\begin{proof} One has $\d z/\d t=1/\sqrt{6}$.
Moreover, there exists a unique $\xi_0$ in $(0,\omega)$ such that
$\wp(\xi_0+\omega')=0$ (with $\wp'(\xi_0+\omega')<0$); by symmetry,
$\wp(-\xi_0+\omega')=0$ (with $\wp'(-\xi_0+\omega')>0$), so $\psi(0)=0$ (with
$\psi'(0)>0$) implies $z(0)=-\xi_0+\omega'$, that is $z(t)=t/\sqrt{6}-\xi_0+\omega'$.
As $\psi(1)=0$, necessarily $z(1)=\xi_0+\omega'$, so $\xi_0=1/(2\sqrt{6})$.
\end{proof} 

\noi Recalling Proposition~\ref{prop20}, one eventually gets that the maximal potential
for $q=2$ is
\[ V(t) = \frac{1}{3}\Psi(t) = \frac{1}{3}\wp(\frac{2t-1}{2\sqrt{6}}+\omega') \]
for the unique $H$ in $(9A^2/2,h^*(A))$ such that $\Psi(1)=0$, that is
\[ \wp(\frac{1}{2\sqrt{6}}+\omega') = 0. \]
This proves Theorem~\ref{thm3}.


\begin{thebibliography}{10}

\bibitem{agrachev-2004a} A.~A. Agrachev and Y.~L. Sachkov,
{\em Control Theory from the Geometric Viewpoint}.
Springer, 2004.

\bibitem{aar-r} P. Albin, C.L. Aldana and F. Rochon, Ricci flow and the determinant of the Laplacian on non-compact surfaces,
Comm. Partial Differential Equations {\bf 38} (2013), 711--749.


\bibitem{aursal}
E. Aurell and P. Salomonson,
On functional determinants of Laplacians in polygons and simplicial complexes,
Comm. Math. Phys. {\bf 165} (1994), 233--259.

\bibitem{bonnard-2003a} B. Bonnard and M. Chyba,
{\em Singular trajectories and their role in control theory}.
Springer, 2003.

\bibitem{BFK}
D. Burghelea, L. Friedlander and T. Kappeler,
On the determinant of elliptic boundary value problems on a line segment,
Proc. Amer. Math. Soc. {\bf 123} (1995), 3027-3038.

\bibitem{cesari-1983a} L. Cesari,
{\em Optimization theory and applications}.
Springer, 1983.

\bibitem{CL}
E. A. Coddington and N. Levinson,
{\em Theory of ordinary differential equations},
McGraw-Hill Book Company, Inc., New York-Toronto-London, 1955.

\bibitem{frei}
P. Freitas,
The spectral determinant of the isotropic quantum harmonic oscillator in arbitrary dimensions,
Math. Ann. {\bf 372} (2018), 1081--1101.

\bibitem{geya}
I.M. Gelfand and A.M. Yaglom, Integration in functional spaces and it applications in quantum physics,
J. Math. Phys. {\bf 1} (1960), 48--69.


\bibitem{H}
E. M. Harrell,
Hamiltonian operators with maximal eigenvalues,
J. Math. Phys. {\bf 25} (1984), 48--51; Erratum, J. Math. Phys. {\bf 27} (1986), 419.

\bibitem{Kato}
T. Kato, 
{\em Perturbation Theory for Linear Operators},
Springer-Verlag Berlin Heidelberg, 1995.


\bibitem{Lesch}
M. Lesch, 
Determinants of regular singular Sturm-Liouville operators,
Math. Nachr. 194 (1998) 139-170.

\bibitem{Lesch-Tolk} 
M. Lesch and J. Tolksdorf,
On the determinant of one-dimensional elliptic boundary value problems,
Comm. Math. Phys. {\bf 193} (1998), 643--660.

\bibitem{Levit-Smilansky}
S. Levit and U. Smilansky, 
A theorem of infnite products of eigenvalues of Sturm-Liouville type operators,
Proc. Amer. Math. Soc. {\bf 65} (1977), 299--302.

\bibitem{mipl}
S. Minakshisundaram and \AA. Pleijel,
Some properties of the eigenfunctions of the Laplace-operator on Riemannian manifolds.
Canadian J. Math. {\bf 1} (1949). 242--256. 

\bibitem{ops}
B. Osgood, R. Phillips, R. and P. Sarnak,
Extremals of determinants of Laplacians,
J Funct. Anal. {\bf 80} (1988), 148--211.

\bibitem{rasi}
D.B. Ray and I.M. Singer,
R-torsion and the Laplacian on Riemannian manifolds,
Adv. Math. {\bf 7} (1971), 145--210.

\bibitem{riem}
B. Riemann, Ueber die Anzahl der Primzahlen unter einer gegebenen Gr\"{o}sse,
Monatsber. Berlin. Akad. (1859), 671--680, English translation in H.M. Edwards, {\it Riemann’s zeta function},
Dover Publications, Inc., Mineola, NY, 2001. Reprint of the 1974 original (Academic Press, New York).

\bibitem{titc}
E.C. Titchmarsh, {\it The theory of the Riemann zeta function}, Oxford Science Publications, 2nd edition,
revised by D.R. Heatn-Brown, Clarendon Press, Oxford (1988).

\bibitem{Savchuk}
A.M. Savchuk 
On the eigenvalues and eigenfunctions of the Sturm-Liouville operator with a singular potential. 
Math. Notes {\bf 69} (2001), 277--285. 

\bibitem{SavcShka99}
A. M. Savchuk and A. A. Shkalikov,
Sturm-Liouville operators with singular potentials,
Math. Notes {\bf 66} (1999), 741--753.





\end{thebibliography}
\end{document}